\documentclass[12pt,english]{amsart}

\usepackage[T1]{fontenc}
\usepackage[utf8]{inputenc}
\usepackage[main=english]{babel}

\usepackage{geometry}
\usepackage{amsmath,amsthm,amssymb,amsfonts,mathtools}
\usepackage{mathrsfs}
\usepackage{esint}

\usepackage{graphicx}   
\usepackage{float}
\usepackage{tikz}       

\usepackage{comment}
\usepackage{indentfirst}
\usepackage[hidelinks]{hyperref}

\theoremstyle{definition}
\newtheorem{theorem}{Theorem}[section]
\newtheorem{lemma}[theorem]{Lemma}
\newtheorem{remark}[theorem]{Remark}

\newtheorem{conjecture}[theorem]{Conjecture}
\newtheorem{question}[theorem]{Question}
\title{Minimal spheres and scalar curvature}
\author{Talant Talipov}
\date{}

\begin{document}

\begin{abstract}
    In 1982, S.-T. Yau conjectured that there exist four distinct embedded minimal two-spheres in any manifold diffeomorphic to \(S^3\). Wang--Zhou confirmed this conjecture for Riemannian three-spheres when the metric is bumpy or has positive Ricci curvature. We prove the following quantitative version of their theorem. Suppose that \((S^3,g)\) has positive Ricci curvature and scalar curvature \(R_g\ge \Lambda_0>0\). Then there exist four distinct embedded minimal two-spheres \(\Sigma_1,\ldots,\Sigma_4\subset (S^3,g)\) such that \(\operatorname{area}_{g}(\Sigma_i)\le 12\pi(i+1)/\Lambda_0\) for every \(i=1,\ldots,4\). We apply this result to a problem posed by S.-T. Yau in 1987 on whether the planar two-spheres are the only minimal spheres in ellipsoids centered at the origin in \(\mathbb R^4\). Haslhofer--Ketover proved that ellipsoids with one sufficiently large semi-axis contain at least one non-planar embedded minimal two-sphere. We prove that such ellipsoids contain at least three non-planar embedded minimal two-spheres.
\end{abstract}

\maketitle

\section{Introduction}

In 1982, S.-T. Yau posed the following conjecture.

\begin{conjecture}[{S.-T. Yau \cite[Problem~89]{Y}}]
\label{conj:yau}
    There exist four distinct embedded minimal two-spheres in any manifold diffeomorphic to \(S^3\).
\end{conjecture}

Simon--Smith \cite{SS} proved the existence of at least one embedded minimal two-sphere in \(S^3\) with an arbitrary metric; see also Colding--De Lellis \cite{CD}, De Lellis--Pellandini \cite{DP}, and Ketover \cite{K}. Later, White \cite{W1}, using degree methods, proved the existence of at least two embedded minimal two-spheres when the metric has positive Ricci curvature, and at least four embedded minimal two-spheres when the metric is sufficiently close to the round metric. Using the catenoid estimate of Ketover--Marques--Neves \cite{KMN}, Haslhofer--Ketover \cite{HK} proved the existence of at least two embedded minimal two-spheres for bumpy metrics. We also remark that branched immersed minimal two-spheres were obtained by Sacks--Uhlenbeck \cite{SU} using min-max theory for harmonic maps; see also Colding--Minicozzi \cite{CM}. In a breakthrough paper, Wang--Zhou \cite{WZ} proved the existence of at least four embedded minimal two-spheres when the metric is bumpy or has positive Ricci curvature, thus confirming Conjecture \ref{conj:yau} in these cases.

The main result of this paper is the following.

\begin{theorem}
\label{thm:main}
    Let \(\Lambda_0>0\). Suppose that \((S^3,g_0)\) satisfies \(\operatorname{Ric}_{g_0}>0\) and \(R_{g_0}\ge \Lambda_0\). Then there exist four distinct embedded minimal two-spheres
    \[
    \Sigma_1,\ldots,\Sigma_4 \subset (S^3,g_0)
    \]
    such that, for every \(i=1,\ldots,4\),
    \[
    \operatorname{area}_{g_0}(\Sigma_i)\le \frac{12\pi(i+1)}{\Lambda_0}.
    \]
\end{theorem}

This gives a quantitative version of the existence theorem of Wang--Zhou \cite{WZ}. Our proof uses Simon--Smith min-max theory \cite{SS}, Hamilton's Ricci flow \cite{H} and the sharp eigenvalue inequality of Karpukhin--Nadirashvili--Penskoi--Polterovich for metrics on \(S^2\) \cite{KNPP}.

For \(i=1\), the corresponding sharp inequality was proved by Marques--Neves \cite{MN}, with equality if and only if the metric has constant sectional curvature. For \(i=2\), a similar estimate, with a larger constant, can be obtained by combining results of Marques--Neves \cite{MN} and Haslhofer--Ketover \cite{HK}. For related results connecting areas of minimal spheres with the volume of the ambient three-sphere, see Ambrozio--Montezuma \cite{AM}.

\subsection*{Applications.}
We apply Theorem \ref{thm:main} in the following setting. Given $a> b> c> d>0$, consider the ellipsoid
\[
E(a,b,c,d):= \left\{\frac{x_1^2}{a^2}+ \frac{x_2^2}{b^2}+\frac{x_3^2}{c^2}+\frac{x_4^2}{d^2}=1\right\}\subset\mathbb{R}^4.
\]
It contains at least four minimal planar two-spheres, which are obtained by the intersection of $E(a,b,c,d)$ with the coordinate hyperplanes $\{x_i=0\}$.

In 1987, S.-T. Yau asked the following question.

\begin{question}[{S.-T. Yau \cite[Section~4]{Y1}}]
\label{ques:y}
    Are the only minimal two-spheres in an ellipsoid centered about the origin in $\mathbb{R}^4$ the planar ones?
\end{question}

Almgren \cite{A} proved that in round spheres, the only immersed minimal two-spheres are planar.  White \cite{W} proved that when \(a\), \(b\), \(c\), and \(d\) are sufficiently close to each other, all minimal two-spheres in \(E(a,b,c,d)\) are planar. In \cite{HK}, Haslhofer--Ketover proved that for fixed \(b,c\) and \(d\), if \(a\) is chosen sufficiently large, then the ellipsoid \(E(a,b,c,d)\) contains a non-planar embedded minimal two-sphere, thus obtaining a negative answer to Question \ref{ques:y}. Bettiol--Piccione \cite{BP} proved that if at least two of the semi-axes \(b,c,d\) are equal, then there are arbitrarily many distinct non-planar embedded minimal two-spheres in \(E(a,b,c,d)\), provided \(a\) is sufficiently large. As a corollary of Theorem \ref{thm:main}, we further address Question \ref{ques:y} and obtain the following result. 

\begin{theorem}
\label{thm:yau}
For fixed $b,c$ and $d$, if $a$ is chosen sufficiently large, then the ellipsoid $E(a,b,c,d)$ contains at least three non-planar embedded minimal two-spheres.
\end{theorem}

\subsection*{Acknowledgments.}
The author is grateful to Yevgeny Liokumovich for his invaluable supervision and encouragement throughout this work. The author is thankful to Brendan Isley for several helpful discussions. This work was partially supported by the Dr. Sergiy and Tetyana Kryvoruchko Graduate Scholarship in Mathematics.

\section{Main result}

We first recall the Simon--Smith min--max theory \cite{SS}; see \cite{Chu, CL, CLhom, CLW1, CLW2, CD, DP, HK, HK1, K, LWY, LW, WZ} for recent developments in the field.

Let \((S^3,g)\) be a Riemannian three-sphere. Consider the spaces
\[
\mathcal X:=\{\phi(S^2)\big|\phi: S^2\to S^3 \text{ is a smooth embedding}\},
\]
and 
\[
\mathcal Y:=\{\phi(S^2)\big|\phi:S^2\to S^3 \text{ is a smooth map whose image is a one-dimensional graph}\}.
\]
Denote by $\mathcal S=\mathcal X\cup \mathcal Y$ endowed with the unoriented smooth topology.  

By Hatcher's proof of Smale's conjecture \cite{Ha}, the space \(\mathcal S\) is homotopy equivalent to the space of geodesic spheres (including degenerate) in the round three-sphere. Hence,
\[
H^*(\mathcal S,\partial\mathcal S;\mathbb Z_2)
\simeq
\mathbb Z_2[\alpha]/(\alpha^5),
\]
where
\[
\alpha\in H^1(\mathcal S,\partial\mathcal S;\mathbb Z_2)
\]
is a generator.

A \textit{sweepout} is a continuous family of elements of \(\mathcal S\),
\[
\Phi:X\to \mathcal S,
\]
where \(X\) is a finite cubical complex and \(\Phi(\partial X)\subset \partial\mathcal S\). For \(i=1,\ldots,4\), we say that a sweepout \(\Phi\) \textit{detects \(\alpha^i\)} if
\[
\Phi^*(\alpha^i)\neq 0
\qquad\text{in } H^i(X,\partial X;\mathbb Z_2).
\]
Let \(\mathcal P_i\) be the collection of sweepouts detecting \(\alpha^i\). We define the \(i\)-th Simon--Smith width by
\[
L_i(g)
=
\inf_{\Phi\in\mathcal P_i}
\sup_{x\in \operatorname{dom}(\Phi)}
\operatorname{area}_g(\Phi(x)).
\]

The Simon--Smith min--max theorem (see Smith \cite{SS}, Colding--De Lellis \cite{CD}, De Lellis--Pellandini \cite{DP} and Ketover \cite{K}) and the multiplicity-one theorem of Wang--Zhou \cite{WZ}, imply that, under positive Ricci curvature, the widths \(L_1(g),\ldots,L_4(g)\) are realized by four distinct embedded minimal two-spheres. Hence Theorem \ref{thm:main} follows from the following width estimate.

\begin{theorem}
\label{thm:higher-ss-width-bound}
Let \(\Lambda_0>0\). Suppose \((S^3,g_0)\) satisfies \(\operatorname{Ric}_{g_0}>0\) and \(R_{g_0}\ge \Lambda_0\).
Then, for every \(i=1,\ldots,4\),
\[
L_i(g_0)\le \frac{12\pi(i+1)}{\Lambda_0}.
\]
\end{theorem}

\begin{remark}
    If one drops the assumption \(\operatorname{Ric}_{g_0}>0\) in Theorem \ref{thm:higher-ss-width-bound}, the statement does not hold even for \(i=1\), as proved by Montezuma \cite{M}.
\end{remark}

Let \(q\ge 4\), and let \(\Gamma^q\) denote the Banach manifold of \(C^q\) Riemannian metrics on \(S^3\).

\begin{lemma}
\label{lem:higher-ss-derivative}
Let \(g:[a,b]\to \Gamma^q\) be a smooth embedding. Then there exist a smooth embedding \(h:[a,b]\to \Gamma^q\), arbitrarily close to \(g\) in the smooth topology, and a subset \(J\subset [a,b]\) of full Lebesgue measure such that the following properties hold.

For every \(i=1,\ldots,4\), the function
\[
t\longmapsto L_i(h(t))
\]
is differentiable at every \(\tau\in J\). Moreover, for every \(\tau\in J\) and every \(i=1,\ldots,4\), there exist pairwise disjoint embedded minimal two-spheres
\[
\Sigma^i_{1},\ldots,\Sigma^i_{N^i}
\subset (S^3,h(\tau))
\]
of class \(C^q\) and positive integers \(m^i_{1},\ldots,m^i_{N^i}\) such that
\[
L_i(h(\tau))
=
\sum_{j=1}^{N^i}
m^i_{j}\operatorname{area}_{h(\tau)}(\Sigma^i_{j}),
\]
\[
\sum_{j=1}^{N^i}
\operatorname{index}_{h(\tau)}(\Sigma^i_{j})
\le i,
\]
and
\[
\left.\frac{d}{dt}\right|_{t=\tau}L_i(h(t))
=
\frac12
\sum_{j=1}^{N^i}
m^i_{j}
\int_{\Sigma^i_{j}}
\operatorname{tr}_{\Sigma^i_{j}}(\partial_t h(\tau))\,dA_{h(\tau)}.
\]
\end{lemma}

\begin{proof}
The proof follows the arguments in Lemma 2.2.3 of Ambrozio--Montezuma \cite{AM} and Lemma 2 of Marques--Neves--Song \cite{MNS}.

By White's bumpy metric theorem \cite{W1,W} and the Sard--Smale transversality theorem \cite{Sm}, the path \(g\) can be perturbed arbitrarily close in the smooth topology to a smooth embedding
\[
h:[a,b]\to \Gamma^q
\]
such that the set of times \(\tau\in[a,b]\) for which \(h(\tau)\) admits no embedded minimal surface with a nontrivial Jacobi field has full measure.

By the same argument as in Lemma 4.1 in Marques--Neves \cite{MN}, for every fixed \(i=1,\ldots,4\), the function
\[
W_i(t):=L_i(h(t))
\]
is Lipschitz. By Rademacher's theorem, \(W_i\) is differentiable for
almost every \(t\). Intersecting the corresponding differentiability sets for each width with the set of bumpy times, we obtain a full-measure set \(J\subset[a,b]\) such that, for every \(\tau\in J\), the metric \(h(\tau)\) is bumpy and every function \(W_i(t)=L_i(h(t))\) is differentiable at \(\tau\).

Fix \(\tau\in J\) and \(i\in\{1,\ldots,4\}\). Let \(t_\ell\to \tau\), \(t_\ell\ne\tau\).
Since \(W_i\) is differentiable at \(\tau\),
\[
W_i'(\tau)
=
\lim_{\ell\to\infty}
\frac{W_i(t_\ell)-W_i(\tau)}{t_\ell-\tau}.
\]
For each \(\ell\), the Simon--Smith min--max theorem (see Smith \cite{SS}, Colding--De Lellis \cite{CD}, De Lellis--Pellandini \cite{DP} and Ketover \cite{K}) and the Morse index bound (see Theorem 1.2 and Subsection 1.3 in Marques--Neves \cite{MN1}) give an integral stationary varifold
\[
V^i_\ell
=
\sum_{a=1}^{N^i_\ell}
m^i_{\ell,a}|\Sigma^i_{\ell,a}|,
\]
where \(\bigl\{\Sigma^i_{\ell,1},\ldots,\Sigma^i_{\ell,N^i_\ell}\bigr\}\) is a collection of pairwise disjoint embedded minimal two-spheres in
\((S^3,h(t_\ell))\), \(m^i_{\ell,a}\in\mathbb N\), and
\[
W_i(t_\ell)
=
\sum_{a=1}^{N^i_\ell}
m^i_{\ell,a}\operatorname{area}_{h(t_\ell)}(\Sigma^i_{\ell,a}),
\]
with
\[
\sum_{a=1}^{N^i_\ell}
\operatorname{index}_{h(t_\ell)}(\Sigma^i_{\ell,a})
\le i.
\]

Since \(W_i(t_\ell)\to W_i(\tau)\), the masses of the varifolds \(V^i_\ell\) are uniformly bounded from above. The monotonicity formula (see Simon \cite{Si}) gives a uniform positive lower bound for the area
of any closed embedded minimal surface in the metrics \(h(t_\ell)\), for \(\ell\) large. Hence, the number of components \(N^i_\ell\) and the multiplicities \(m^i_{\ell,a}\) are uniformly bounded. Passing to a subsequence and relabeling we may assume that \(N^i_\ell=\widetilde N^i\) and \(m^i_{\ell,a} = m^i_a\) for \(a=1,\ldots,\widetilde N^i\).

By the compactness theorem for embedded minimal surfaces with uniformly bounded area and index (see Sharp \cite{Sh} and Ambrozio--Buzano--Carlotto--Sharp \cite{ABCS}), since \(h(\tau)\) is bumpy, after passing to a further subsequence each sequence \(\Sigma^i_{\ell,a}\) converges smoothly and graphically with multiplicity one to an embedded minimal two-sphere in \((S^3,h(\tau))\). The limiting spheres are pairwise disjoint or equal by the maximum principle. After collecting equal limits and adding the corresponding multiplicities, we obtain pairwise disjoint embedded minimal two-spheres
\[
\Sigma^i_{1},\ldots,\Sigma^i_{N^i}
\subset (S^3,h(\tau))
\]
and positive integers \(m^i_{1},\ldots,m^i_{N^i}\) such that
\[
W_i(\tau)
=
\sum_{j=1}^{N^i}
m^i_{j}\operatorname{area}_{h(\tau)}(\Sigma^i_{j}),
\]
and
\[
\sum_{j=1}^{N^i}
\operatorname{index}_{h(\tau)}(\Sigma^i_{j})
\le i.
\]

It remains to compute the derivative. Since \(h(\tau)\) is bumpy, each limiting
sphere \(\Sigma^i_j\) is nondegenerate. By the smooth graphical convergence above, the minimal surface equation, and the invertibility of the Jacobi operator of \(\Sigma^i_j\), we have
\[
\operatorname{area}_{h(t_\ell)}(\Sigma^i_{\ell,a})
-
\operatorname{area}_{h(\tau)}(\Sigma^i_j)
=
\frac{t_\ell-\tau}{2}
\int_{\Sigma^i_j}
\operatorname{tr}_{\Sigma^i_j}(\partial_t h(\tau))\,dA_{h(\tau)}
+
o(|t_\ell-\tau|)
\]
whenever $\Sigma^i_{\ell,a} \to \Sigma^i_j$. Thus,
\[
\begin{aligned}
W_i'(\tau)
&=
\lim_{\ell\to\infty}
\frac{W_i(t_\ell)-W_i(\tau)}{t_\ell-\tau} \\
&=
\frac12
\sum_{j=1}^{N^i}
m^i_j
\int_{\Sigma^i_j}
\operatorname{tr}_{\Sigma^i_j}(\partial_t h(\tau))\,dA_{h(\tau)}.
\end{aligned}
\]
This completes the proof.
\end{proof}

\begin{lemma}
\label{lem:integrated-ricci-derivative}
Let \(g(t)\), \(t\in[a,b]\), be a smooth solution of the unnormalised Ricci flow on \(S^3\),
\[
\partial_t g(t)=-2\operatorname{Ric}_{g(t)}.
\]
Assume that
\[
\operatorname{Ric}_{g(t)}>0
\]
for every \(t\in[a,b]\). Then, for every \(i=1,\ldots,4\) and every \(a\le t_1<t_2\le b\),
\[
L_i(g(t_2))-L_i(g(t_1))
\ge
-8\pi(i+1)(t_2-t_1).
\]
\end{lemma}

\begin{proof}
Fix $i \in \{1,\ldots,4\}$. Write
\[
W_i(t)=L_i(g(t)).
\]
We first prove the estimate under the additional assumption that the conclusion of Lemma \ref{lem:higher-ss-derivative} holds for the path \(g(t)\) on a subset $J \subset [a,b]$ of full measure.

Fix $t \in J$. Then, since \(\operatorname{Ric}_{g(t)}>0\), by the Frankel property and the multiplicity-one theorem of
Wang--Zhou \cite{WZ}, there exists an embedded minimal sphere
\[
\Gamma_t\subset (S^3,g(t))
\]
such that
\[
W_i(t)=\operatorname{area}_{g(t)}(\Gamma_t),
\]
\[
\operatorname{index}_{g(t)}(\Gamma_t)\le i,
\]
and
\[
W_i'(t)
=
\frac12
\int_{\Gamma_t}
\operatorname{tr}_{\Gamma_t}(\partial_t g(t))\,dA_{g(t)}.
\]

Since \(\partial_t g=-2\operatorname{Ric}\), we have
\[
W_i'(t)
=
-\int_{\Gamma_t}
\operatorname{tr}_{\Gamma_t}\operatorname{Ric}\,dA
=
-\int_{\Gamma_t}
\left(R-\operatorname{Ric}(\nu,\nu)\right)\,dA.
\]
By the Gauss equation and the minimal surface equation,
\[
R-\operatorname{Ric}(\nu,\nu)
=
2K_{\Gamma_t}
+
\operatorname{Ric}(\nu,\nu)
+
|A|^2.
\]
Since \(\Gamma_t\simeq S^2\), the Gauss--Bonnet formula gives
\[
\int_{\Gamma_t}2K_{\Gamma_t}\,dA=8\pi.
\]
Thus
\[
W_i'(t)
=
-8\pi
-
\int_{\Gamma_t}
\left(\operatorname{Ric}(\nu,\nu)+|A|^2\right)\,dA.
\]

It remains to estimate the last integral. Let
\[
q=\operatorname{Ric}(\nu,\nu)+|A|^2.
\]
Since \(\operatorname{Ric}>0\), we have \(q>0\). Let
\[
m=\operatorname{index}_{g(t)}(\Gamma_t)\le i.
\]
Let \(u\) be a nonzero solution of the following weighted Laplace eigenvalue problem
\begin{equation}
\label{eq:eigenvalue}
    -\Delta u=\lambda q u
\end{equation}
on \(\Gamma_t\). The second variation formula gives
\[
Q_{\Gamma_t}(u,u)
=
\int_{\Gamma_t}|\nabla u|^2\,dA
-
\int_{\Gamma_t}q u^2\,dA.
\]
Integration by parts gives
\[
\int_{\Gamma_t}|\nabla u|^2\,dA
=
\lambda
\int_{\Gamma_t}q u^2\,dA.
\]
Hence
\[
Q_{\Gamma_t}(u,u)
=
(\lambda-1)\int_{\Gamma_t}q u^2\,dA.
\]
Thus \(Q_{\Gamma_t}\) is negative on the span of the solutions of \eqref{eq:eigenvalue} with \(\lambda<1\). Hence, the number of eigenvalues \(\lambda<1\), counted with multiplicity, is at most $m=\operatorname{index}_{g(t)}(\Gamma_t)$.

Let
\[
0=\lambda_0<\lambda_1\le\lambda_2\le\ldots
\]
be the weighted spectrum for \eqref{eq:eigenvalue}. Then
\[
\lambda_m\ge 1.
\]

Now define a conformal metric on \(\Gamma_t\) by
\[
\widehat g=q\,g_{\Gamma_t}.
\]
Since $\Gamma_t$ has dimension two, the Dirichlet energy is conformally invariant. Moreover,
\[
\int_{\Gamma_t}u^2\,dA_{\widehat g}
=
\int_{\Gamma_t}q u^2\,dA,
\]
and hence the weighted Laplace eigenvalue problem above is exactly the ordinary Laplace eigenvalue problem for \((S^2,\widehat g)\). Also,
\[
\operatorname{area}_{\widehat g}(S^2)
=
\int_{\Gamma_t}q\,dA.
\]
By the Karpukhin--Nadirashvili--Penskoi--Polterovich sharp isoperimetric inequality for Laplace eigenvalues on \(S^2\) \cite{KNPP} (see also Hersch \cite{He}, Nadirashvili \cite{N}, Petrides \cite{P} and Nadirashvili--Sire \cite{NS}),
\[
\lambda^{\operatorname{Lap}}_m(\widehat g)\operatorname{area}_{\widehat g}(S^2)\le 8\pi m.
\]
Since \(\lambda^{\operatorname{Lap}}_m(\widehat g) = \lambda_m \ge 1\), we obtain
\[
\int_{\Gamma_t}
\left(\operatorname{Ric}(\nu,\nu)+|A|^2\right)\,dA
=
\int_{\Gamma_t}q\,dA
\le 8\pi m
\le 8\pi i.
\]
Consequently, for almost every \(t\),
\[
W_i'(t)
\ge
-8\pi-8\pi i
=
-8\pi(i+1).
\]
Integrating from \(t_1\) to \(t_2\) gives
\[
W_i(t_2)-W_i(t_1)
\ge
-8\pi(i+1)(t_2-t_1).
\]

We now remove the additional assumption. Applying Lemma \ref{lem:higher-ss-derivative}, we obtain smooth paths \(g_\ell(t)\)
for \(t\in[a,b]\) and \(\ell \in \mathbb{N}\) such that \(g_\ell\to g\) in \(C^1_tC^\infty_x\), and such that the conclusion of Lemma \ref{lem:higher-ss-derivative} holds for \(g_\ell(t)\) on full-measure sets \(J_\ell\subset[a,b]\). We can also assume, after taking
\(\ell\) sufficiently large, that \(\operatorname{Ric}_{g_\ell(t)}>0\)
for every \(t\in[a,b]\).

Set
\[
W_i^\ell(t)=L_i(g_\ell(t)).
\]
Since \(g_\ell\to g\) in the smooth topology and \(g(t)\) solves
\[
\partial_t g(t)=-2\operatorname{Ric}_{g(t)},
\]
we have
\[
E_\ell(t):=\partial_t g_\ell(t)+2\operatorname{Ric}_{g_\ell(t)}
\to 0
\]
uniformly on \([a,b]\times S^3\).

Fix \(t\in J_\ell\). By the conclusion of Lemma \ref{lem:higher-ss-derivative}, the Frankel property, and the multiplicity-one theorem of Wang--Zhou \cite{WZ}, there exists an embedded minimal
two-sphere
\[
\Gamma_{\ell,t}\subset (S^3,g_\ell(t))
\]
such that
\[
W_i^\ell(t)=\operatorname{area}_{g_\ell(t)}(\Gamma_{\ell,t}),
\]
\[
\operatorname{index}_{g_\ell(t)}(\Gamma_{\ell,t})\le i,
\]
and
\[
(W_i^\ell)'(t)
=
\frac12
\int_{\Gamma_{\ell,t}}
\operatorname{tr}_{\Gamma_{\ell,t}}(\partial_t g_\ell(t))\,dA_{g_\ell}.
\]

Since
\[
\partial_t g_\ell
=
-2\operatorname{Ric}_{g_\ell}+E_\ell,
\]
we have
\[
(W_i^\ell)'(t)
=
-\int_{\Gamma_{\ell,t}}
\left(R_{g_\ell}-\operatorname{Ric}_{g_\ell}(\nu,\nu)\right)\,dA_{g_\ell}
+
\frac12
\int_{\Gamma_{\ell,t}}
\operatorname{tr}_{\Gamma_{\ell,t}}E_\ell\,dA_{g_\ell}.
\]
The first term is bounded from below by \(-8\pi(i+1)\) by the argument proved above.

It remains to estimate the second term. Define
\[
\varepsilon_\ell
=
\frac12
\sup_{\substack{t\in[a,b],\,x\in S^3\\
P\subset T_xS^3,\ \dim P=2}}
\left|
\operatorname{tr}_{P,g_\ell(t)} E_\ell(t)
\right|.
\]
Since \(E_\ell\to0\) uniformly, we have
\[
\varepsilon_\ell\to0.
\]
Also,
\[
\frac12
\int_{\Gamma_{\ell,t}}
\operatorname{tr}_{\Gamma_{\ell,t}}E_\ell\,dA_{g_\ell}
\ge
-\varepsilon_\ell
\operatorname{area}_{g_\ell(t)}(\Gamma_{\ell,t})
=
-\varepsilon_\ell W_i^\ell(t).
\]
Thus, for every \(t\in J_\ell\),
\[
(W_i^\ell)'(t)
\ge
-8\pi(i+1)-\varepsilon_\ell W_i^\ell(t).
\]

Since \(g_\ell(t)\to g(t)\) uniformly in the smooth topology on the compact interval \([a,b]\), and the widths are continuous with respect to the smooth topology, there exists a constant \(C>0\), independent of \(\ell\), such that
\[
W_i^\ell(t)\le C
\]
for all \(t\in[a,b]\) and all sufficiently large \(\ell\). Then
\[
(W_i^\ell)'(t)
\ge
-8\pi(i+1)-C\varepsilon_\ell
\]
for almost every \(t\in[a,b]\). Integrating from \(t_1\) to \(t_2\) gives
\[
W_i^\ell(t_2)-W_i^\ell(t_1)
\ge
-\left(8\pi(i+1)+C\varepsilon_\ell\right)(t_2-t_1).
\]
Finally, \(W_i^\ell(t)\to W_i(t)\) uniformly on \([a,b]\). Letting
\(\ell\to\infty\), we obtain
\[
W_i(t_2)-W_i(t_1)
\ge
-8\pi(i+1)(t_2-t_1).
\]
This proves the lemma.
\end{proof}

\begin{proof}[Proof of Theorem ~\ref{thm:higher-ss-width-bound}]
Let \(g(t)\), \(t\in[0,T)\), be the unnormalised Ricci flow starting from \(g_0\):
\[
\partial_t g(t)=-2\operatorname{Ric}_{g(t)},
\qquad
g(0)=g_0.
\]
By Hamilton's theorem \cite{H}, positive Ricci curvature is preserved and the solutions converge to a round point as \(t \to T\). Then
\[
\lim_{t\to T}L_i(g(t))=0.
\]

By the evolution equation of the scalar curvature
\[
\partial_t R=\Delta R+2|\operatorname{Ric}|^2.
\]
In dimension three,
\[
|\operatorname{Ric}|^2\ge \frac13R^2.
\]
Hence,
\[
\partial_t R\ge \Delta R+\frac23R^2.
\]
Let \(\rho(t)\) be the solution of
\[
\rho'(t)=\frac23\rho(t)^2,
\qquad
\rho(0)=\Lambda_0.
\]
Explicitly,
\[
\rho(t)
=
\frac{\Lambda_0}{1-\frac23\Lambda_0 t}.
\]
Since \(R_{g_0}\ge \Lambda_0=\rho(0)\), the parabolic maximum principle gives
\[
R_{g(t)}\ge \rho(t)
=
\frac{\Lambda_0}{1-\frac23\Lambda_0 t}
\]
for all
\[
0\le t<\min\left\{T,\frac{3}{2\Lambda_0}\right\}.
\]
Thus the flow cannot exist smoothly past \(t=3/(2\Lambda_0)\). Hence
\[
T\le \frac{3}{2\Lambda_0}.
\]

Fix \(i\in\{1,\ldots,4\}\). By Lemma \ref{lem:integrated-ricci-derivative}, for every
\(0\le s<T\),
\[
L_i(g(s))-L_i(g_0)
\ge
-8\pi(i+1)s.
\]
Equivalently,
\[
L_i(g_0)
\le
L_i(g(s))+8\pi(i+1)s.
\]
Taking \(s\to T\), we obtain
\[
L_i(g_0)
\le
8\pi(i+1)T.
\]
Since \(T\le 3/(2\Lambda_0)\), this yields
\[
L_i(g_0)
\le
8\pi(i+1)\cdot \frac{3}{2\Lambda_0}
=
\frac{12\pi(i+1)}{\Lambda_0}.
\]
This proves the theorem.
\end{proof}

\begin{proof}[Proof of Theorem~\ref{thm:yau}]
Fix \(b>c>d>0\), and write
\[
E_a:=E(a,b,c,d).
\]
Let \(g_a\) denote the metric induced on \(E_a\) by the Euclidean metric of
\(\mathbb R^4\). For \(j=1,\ldots,4\), denote the coordinate planar minimal
two-spheres by
\[
P_j(a):=E_a\cap\{x_j=0\}.
\]

We first record a scalar curvature lower bound for \(E_a\), uniform for all
\(a>b\). Let
\[
A_a=\operatorname{diag}(a^{-2},b^{-2},c^{-2},d^{-2}).
\]
Then
\[
E_a=\{x\in\mathbb R^4:\langle A_a x,x\rangle=1\}.
\]
For \(v\in T_xE_a\), the second fundamental form is
\[
\mathrm{II}(v,v)=\frac{\langle A_a v,v\rangle}{|A_a x|}.
\]
Moreover,
\[
|A_a x|^2
=
\langle A_a^2x,x\rangle
\le
d^{-2}\langle A_a x,x\rangle
=
d^{-2}.
\]
Hence,
\[
|A_a x|\le d^{-1}.
\]

Let \(H_1=\{v_1=0\}\subset\mathbb R^4\). Since
\(\dim T_xE_a=3\) and \(\dim H_1=3\), we have
\[
\dim(T_xE_a\cap H_1)\ge2.
\]
For every \(v\in T_xE_a\cap H_1\),
\[
\langle A_a v,v\rangle
\ge
b^{-2}|v|^2.
\]
Then
\[
\mathrm{II}(v,v)
\ge
\frac{b^{-2}}{|A_a x|}|v|^2
\ge
\frac{d}{b^2}|v|^2.
\]
By the min-max characterization of eigenvalues, at least two principal curvatures
of \(E_a\) at \(x\) are bounded below by \(d/b^2\). Since \(E_a\) is strictly convex,
all principal curvatures are positive. Thus the Gauss equation gives
\[
R_{g_a}
=
2\sum_{1\le \alpha<\beta\le3}\kappa_\alpha\kappa_\beta
\ge
\frac{2d^2}{b^4} =: \Lambda_0.
\]
Also, since all principal curvatures of \(E_a\) are positive, \(g_a\) has positive
Ricci curvature.

Applying Theorem~\ref{thm:main} to \((E_a,g_a)\), we obtain four distinct embedded minimal two-spheres
\[
\Sigma_1(a),\ldots,\Sigma_4(a)\subset E_a
\]
such that
\[
\operatorname{area}_{g_a}(\Sigma_i(a))
\le
\frac{12\pi(i+1)}{\Lambda_0}
\le
\frac{60\pi}{\Lambda_0}
=
\frac{30\pi b^4}{d^2}
\]
for every \(i=1,\ldots,4\).

We now compare this uniform area bound with the areas of the planar spheres. The planar sphere \(P_1(a)\) is independent of \(a\), and hence has bounded area. On the other hand, for \(j=2,3,4\), the planar sphere \(P_j(a)\) is a two-dimensional ellipsoid with one semi-axis equal to \(a\). Consider the restriction of the coordinate function \(x_1\) to \(P_j(a)\).
Since \(|\nabla^{P_j(a)}x_1|\le1\), by the coarea formula,
\[
\operatorname{area}(P_j(a))
\ge
\int_{-a/2}^{a/2}
\mathcal H^1\bigl(P_j(a)\cap\{x_1=t\}\bigr)\,dt.
\]
For \(|t|\le a/2\), the curve \(P_j(a)\cap\{x_1=t\}\) is an ellipse with smaller semi-axis at least
\[
d\sqrt{1-\frac{t^2}{a^2}}
\ge
\frac{\sqrt3}{2}d.
\]
Therefore its length is at least \(\sqrt3\pi d\), and then
\[
\operatorname{area}(P_j(a))
\ge
\sqrt3\pi da
\]
for \(j=2,3,4\). Thus, for \(a\) sufficiently large,
\[
\operatorname{area}(P_j(a))
>
\frac{30\pi b^4}{d^2}
\]
for \(j=2,3,4\).

For such \(a\), none of the four minimal two-spheres
\(\Sigma_1(a),\ldots,\Sigma_4(a)\) can be equal to \(P_2(a),P_3(a)\), or \(P_4(a)\), since each \(\Sigma_i(a)\) has area at most \(30\pi b^4/d^2\), while each of \(P_2(a),P_3(a),P_4(a)\) has strictly larger area. Since the four spheres \(\Sigma_1(a),\ldots,\Sigma_4(a)\) are distinct, at most one of them can coincide with the remaining planar sphere \(P_1(a)\). Hence at least three of the four minimal spheres \(\Sigma_1(a),\ldots,\Sigma_4(a)\) are non-planar.
\end{proof}

\vspace{0.5cm} 
\noindent Department of Mathematics, University of Toronto, Toronto, Canada\\
\textit{E-mail address}: \texttt{talant.talipov@mail.utoronto.ca}

\end{document}